\documentclass[12pt,leqno]{amsart}
\usepackage{amsmath, amssymb, amscd, amsfonts, xypic,   stmaryrd, turnstile, mathrsfs, eucal,color}

\definecolor{dullmagenta}{rgb}{0.4,0,0.4}

\definecolor{darkblue}{rgb}{0,0,0.4}

\usepackage[colorlinks=true, pdfstartview=FitV, linkcolor=red, citecolor=blue, urlcolor=darkblue]
{hyperref}

\newtheorem{theorem}{Theorem}[section]
\newtheorem{lemma}[theorem]{Lemma}
\newtheorem{proposition}[theorem]{Proposition}
\newtheorem{corollary}[theorem]{Corollary}

\theoremstyle{definition}
\newtheorem{definition}[theorem]{Definition}
\newtheorem{example}[theorem]{Example}
\newtheorem{remark}[theorem]{Remark}

\def\car#1,#2,#3,#4{
$$
   \CD
   #1           @>{}>>          #2        \\
   @V{{}}VV                  @VV{{}}V  \\
   #3        @>{{}}>>   #4
   \endCD
   $$}

\def\Car#1,#2,#3,#4,#5,#6,#7,#8{
$$
   \CD
   #1           @>#2>>          #3        \\
   @V{#4}VV                  @VV{#5}V  \\
   #6        @>{#7}>>   #8
   \endCD
   $$}

\begin{document}

\title[catenarian  extensions]{Catenarian  FCP ring extensions}

\author[G. Picavet and M. Picavet]{Gabriel Picavet and Martine Picavet-L'Hermitte}
\address{Math\'ematiques \\
8 Rue du Forez, 63670 - Le Cendre\\
 France}
\email{picavet.mathu (at) orange.fr}

\begin{abstract} If $R\subseteq S$ is a ring extension of commutative unital rings, the poset $[R,S]$ of $R$-subalgebras of $S$ is called catenarian if it verifies the Jordan-H\"older property. This property has already been studied by Dobbs and Shapiro for  finite extensions of fields. We investigate this property for arbitrary ring extensions, showing that many type of extensions are catenarian. We reduce the characterization of catenarian extensions to the case of field extensions, an unsolved question at that time. 
\end{abstract}


\keywords  {FIP, FCP extension, minimal extension, integral extension, support of a module, t-closure, lattice,  catenarian extension, Jordan-H\"older property, distributive extension, algebraic field extension, Galois extension}

\maketitle

\section{Introduction}

In this paper, we  consider the category of commutative and unital rings, whose    epimorphisms will be involved. If  $R\subseteq S$ is a (ring) extension, we denote  by $[R,S]$ the set of all $R$-subalgebras of $S$ and
  set $]R,S[: =[R,S]\setminus \{R,S\}$ (with a similar definition for $[R,S[$ or $]R,S]$). 
 
  A {\it lattice} is a poset $L$ such that every pair $a,b\in L$ has a supremum and an infimum.  We will consider 
    for an extension $R\subseteq S$, the poset $([R,S],\subseteq)$, which  is a {\it complete} lattice where the supremum of any non-void subset  is the compositum which we call {\it product} from now on and denote by $\Pi$ when necessary, and the infimum of any non-void subset is the intersection.   As a general rule, an extension $R\subseteq S$ is said to have some property of lattices if $[R,S]$ has this property. Any undefined  material
 is explained at the end of the section or in the next sections. We  only consider extensions of finite length, in a sense defined  below and called FCP extensions.  These extensions $R\subseteq S$ are the extensions such that $[R,S]$ is an Artinian and Noetherian lattice.
    
     Dobbs and Shapiro published a paper in which they examine finite field extensions with the catenarian property \cite{DS2}. These extensions are such that the lattice $[K,L]$ associated to a field extension $K\subseteq L$ has the Jordan-H\"older property {\it i.e.}, all finite  maximal chains of subfields going from $K$ to $L$ have the same cardinalities. Beyond many interesting  results, a complete characterization of these extensions is lacking and actually is an  open problem, that we did not try to solve. Our aim  is  to consider  arbitrary FCP extensions that are catenarian ({\it i.e.} that have the Jordan-H\"older property). Surprisingly, we are able to give substantial  positive results. Some types of ring extensions are naturally catenarian. 
    For example, in the FCP context, distributive extensions,  length $2$ extensions \cite{Pic 6}, Pr\"ufer extensions \cite{Pic 5},  some pointwise extensions \cite{CPP}, infra-integral extensions are catenarian. The catenarian property   has a good stability with respect to the usual constructions of commutative algebra.  An FCP ring extension $R\subseteq S$ is catenarian if and only if $R \subseteq \overline R$ is catenarian, where $\overline R$ is the integral closure of $R$ in $S$. However, there are some   ring extensions, namely the $t$-closed extensions  whose definition is given in the sequel, that are catenarian if and only if some associated  residual field extensions are catenarian. As  $t$-closed extensions  factorize any ring extension because any extension $R\subseteq S$ has a $t$-closure   ${}_S^tR$ of $R$ in $S$,  the reader may understand that a complete characterization  is unavoidable, unless the same problem is solved for fields.  Actually, some results are valid under the hypothesis ${}_S^tR \subseteq S$ is catenarian, for example if $[R,S]=[R,{}_S^tR]\cup[ {}_S^tR, S]$. 
    
     We give some complements on finite field extensions that are catenarian. Catenarian Galois extensions have been characterized in \cite{DS2}.
    
    Another lattice property is involved: the supersolvable property, a property considered by Dobbs and Shapiro   in the context of Galois group of a field extension, providing finite Galois extensions that are catenarian. Supersolvable lattices are defined in \cite[Example 3.14.4]{St}.  For Galois extensions, the two notions coincide. For arbitrary ring extensions, this is not the case. We must add some special condition of modularity, to get catenarity.
  
  This paper is the continuation of our earlier  papers on lattice properties of ring extensions, \cite{Pic 10} and \cite{Pic 11}, where we considered Boolean ring extensions and Loewy series of a ring extension.  We have a forthcoming paper about distributive extensions.    
     
  \section{ Some material of commutative algebra} 
  
 \subsection{Some conventions and notation}  
   A {\it local} ring is here what is called elsewhere a quasi-local ring. As usual, Spec$(R)$ and Max$(R)$ are the set of prime and maximal ideals of a ring $R$. 
  For an extension $R\subseteq S$ and an ideal $I$ of $R$, we write $\mathrm{V}_S(I):=\{P\in\mathrm{Spec }(S)\mid I\subseteq P\}$. 
    We denote by $\kappa_R(P)$ the residual field $R_P/PR_P$ at $P$. 
  The support of an $R$-module $E$ is $\mathrm{Supp}_R(E):=\{P\in\mathrm{Spec }(R)\mid E_P\neq 0\}$, and $\mathrm{MSupp}_R(E):=\mathrm{Supp}_R(E)\cap\mathrm{Max}(R)$. When $R 
 \subseteq S$ is an extension, we will set   $\mathrm{Supp}_R(T/R):= \mathrm{Supp}(T/R)$  and $\mathrm{Supp}_R(S/T):= \mathrm{Supp}(S/T)$ for each $T\in [R,S]$, unless otherwise specified. 
 
 If $R\subseteq S$ is a ring extension and $P\in\mathrm{Spec}(R)$, then $S_P$ is both the localization $S_{R\setminus P}$ as a ring and the localization at $P$ of the $R$-module $S$. We denote by $(R:S)$ the conductor of $R\subseteq S$. The integral closure of $R$ in $S$ is denoted by $\overline R^S$ (or by $\overline R$ if no confusion can occur). Finally,  $|X|$  is the cardinality of a set $X$,   $\subset$ denotes proper inclusion and, for a positive integer $n$, we set $\mathbb{N}_n:=\{1,\ldots,n\}$.  
    The characteristic of an integral domain $k$ is denoted by $\mathrm{c}(k)$.
  
     We now describe the material we use in this paper.

    \subsection{Results on FCP extensions} 
    
The extension $R\subseteq S$ is said to have FIP (for the ``finitely many intermediate algebras property") or an FIP  extension if $[R,S]$ is finite. 
We will say that $R\subseteq S$ is {\it chained} if $[R,S]$ is a {\it chain}, that is, a set of elements  that are pairwise comparable with respect to inclusion. We  also say that the extension $R\subseteq S$ has FCP  (or is an FCP  extension) if each chain in $[R,S]$ is finite. Clearly,  each extension that satisfies FIP must also satisfy FCP. 
Dobbs and the authors characterized FCP and FIP extensions \cite{DPP2}.

Our main tool are the minimal (ring) extensions, a concept that was introduced by Ferrand-Olivier \cite{FO}. Recall that an extension $R\subset S$ is called {\it minimal} if $[R, S]=\{R,S\}$. 
 The key connection between the above ideas is that if $R\subseteq S$ has FCP, then any maximal (necessarily finite) chain $\mathcal C$ of $R$-subalgebras of $S$, $R=R_0\subset R_1\subset\cdots\subset R_{n-1}\subset R_n=S$, with {\it length} $\ell(\mathcal C):=n <\infty$, results from juxtaposing $n$ minimal extensions $R_i\subset R_{i+1},\ 0\leq i\leq n-1$. 
An FCP extension is finitely generated, and  (module) finite if integral.
For any extension $R\subseteq S$, the {\it length} $\ell[R,S]$ of $[R,S]$ is the supremum of the lengths of chains of $R$-subalgebras of $S$. Notice  that if $R\subseteq S$ has FCP, then there {\it does} exist some maximal chain of $R$-subalgebras of $S$ with length $\ell[R,S]$ \cite[Theorem 4.11]
{DPP3}.   

The following 
results are  deeply involved in the sequel.  A ring extension $R\subseteq S$, defining  $S$ as an $R$-module, is called finite  if the $R$-module $S$ is of finite type.

\begin{theorem}\label{crucial}\cite[Th\'eor\`eme 2.2]{FO} and \cite[Proposition 4.6]{DPPS} 
\begin{enumerate}
\item A  minimal extension  
is either  finite or a flat epimorphism (of finite type).
\item A ring extension $R\subset S$  is minimal if and only if there is some $M\in\mathrm{Max}(R)$, called the {\it crucial} maximal ideal of the extension 
 and denoted by $\mathcal{C}(R,S)$,  
such that 
$R_M\subset S_M$ is minimal and $R_P=S_P$ for each $P\in\mathrm{Spec}(R)\setminus\{M\}$. In particular, $\mathrm{Supp}(S/R)= \{M\}$.
\end{enumerate}
\end{theorem} 

\begin{proposition}\label{1.14} 
 \cite[Corollary 3.2]{DPP2} 
Suppose there exists a maximal chain $R=R_0 \subset\cdots\subset  R_i \subset\cdots \subset R_n=S$ of extensions, where $R_i\subset R_{i+1}$ is minimal with crucial ideal $M_i$. $($For instance, suppose $R \subseteq S$ is an $\mathrm{FCP}$ extension.$)$ Then $\mathrm {Supp}(S/R)=\{M_i\cap R\mid i=0,\ldots,n-1\}$.
\end{proposition}

 Recall  that an  extension $R\subseteq S$ is {\it Pr\"ufer} (or  a normal pair) if $R\subseteq T$ is a flat epimorphism for each $T\in[R,S]$ \cite[Theorem 5.2, page 47]{KZ}. In \cite{Pic 5}, we called a minimal extension which is a flat epimorphism, a {\it Pr\"ufer minimal} extension. 
 Three types of integral minimal extensions exist, characterized in the next theorem, (a consequence of  the fundamental lemma of Ferrand-Olivier), so that there are four types of minimal extensions.
\begin{theorem}\label{minimal} \cite [Theorem 2.2]{DPP2} Let $R\subset T$ be an extension and  $M:=(R: T)$. Then $R\subset T$ is minimal and finite if and only if $M\in\mathrm{Max}(R)$ and one of the following three conditions holds:

\noindent (a) {\bf inert case}: $M\in\mathrm{Max}(T)$ and $R/M\to T/M$ is a minimal field extension.

\noindent (b) {\bf decomposed case}: There exist $M_1,M_2\in\mathrm{Max}(T)$ such that $M= M _1\cap M_2$ and the natural maps $R/M\to T/M_1$ and $R/M\to T/M_2$ are both isomorphisms.

\noindent (c) {\bf ramified case}: There exists $M'\in\mathrm{Max}(T)$ such that ${M'}^2 \subseteq M\subset M',\  [T/M:R/M]=2$, and the natural map $R/M\to T/M'$ is an isomorphism.
\end{theorem}

We recall here the Crosswise Exchange which will be used in the rest of the paper. 

\begin{lemma}\label{1.13} (Crosswise Exchange) \cite[Lemma 2.7]{DPP2} Let $R\subset S$ and $S\subset T$ be  minimal  extensions,  $M:= \mathcal{C}(R,S)$, $N:= \mathcal{C}(S,T)$ and $P:=N\cap R$ be such that $P\not\subseteq M$. Then there is   $S' \in [R,T]$ such that $R\subset S'$  is minimal  of the same type as $S\subset T$ and $P= \mathcal{C}(R,S')$; and $S'\subset T$ is  minimal  of the same type as $R\subset S$ and $MS' =\mathcal{C}(S',T)$. Moreover,  $[R,T]=\{R,S,S',T\}$ and $R_Q=S'_Q=T_Q$ for $Q\in \mathrm{Max}(R)\setminus \{M,P\}$.
\end{lemma}

   The following material is needed for our study.
  
 \begin{definition}\label{1.3} An integral extension $R\subseteq S$ is called {\it infra-integral} \cite{Pic 2} 
  if all its residual extensions $\kappa_R(P)\to \kappa_S(Q)$ 
  (with $Q\in\mathrm {Spec}(S)$ and $P:=Q\cap R$)  
 are isomorphisms. 
 An extension $R\subseteq S$ is called {\it t-closed} (cf. \cite{Pic 2}) if the relations $b\in S,\ r\in R,\ b^2-rb\in R,\ b^3-rb^2\in R$ imply $b\in R$. The $t$-{\it closure} ${}_S^tR$ of $R$ in $S$ is the smallest element $B\in [R,S]$  such that $B\subseteq S$ is t-closed and the greatest element  $B'\in [R,S]$ such that $R\subseteq B'$ is infra-integral. 
 \end{definition}  
 
 Integral closures and  t-closures play an a crucial role in the sequel.

  The next proposition gives the link between 
   these notions 
   and  the minimal extensions involved.

  \begin{proposition}\label{1.31} 
  \cite[Lemma 3.1]{Pic 4} Let there be an integral extension $R\subset S$ admitting  a maximal chain $\mathcal C$ of $ R$-subextensions of $S$, defined by $R=R_0\subset\cdots\subset R_i\subset\cdots\subset R _n= S$, where each $R_i\subset R_{i+1}$ is  minimal.  The following statements hold: 

\begin{enumerate}
\item $R \subset S$ is t-closed if and only if  each $R_i\subset R_{i+1}$ is inert. 

\item  $R\subset S$ is  infra-integral if and only if each  $R_i\subset R_{i+1}$ is either ramified or  decomposed. 

\item If $R$ is field and $R \subset S$ is t-closed,  $S$ is a field.
\end{enumerate}
If  (1) 
holds, then $\mathrm {Spec}(S)\to\mathrm{Spec}(R)$ is bijective. \end{proposition}

\begin{proof} The last result comes from Theorem \ref{minimal} and Theorem \ref{crucial}(2),  because they  imply the bijectivity  of $\mathrm {Spec}(R_{i+1})\to\mathrm{Spec}(R_i)$, when $R_i\subset R_{i+1}$ is 
 inert. 
\end{proof} 

 The following situation is often met when controlling  lattice properties of ring extensions. We say that a property $(\mathcal T)$ of  
  integral FCP $t$-closed extensions $R\subset S$ is {\it convenient} if the following conditions are equivalent.
\begin{enumerate}
\item $R\subset S$ satisfies $(\mathcal T)$. 

\item $R_M\subset S_M$ satisfies $(\mathcal T)$ for any $M\in\mathrm{MSupp}(S/R)$. 

\item  $R/I\subset S/I$ satisfies $(\mathcal T)$ for any  ideal $I$ shared by $R$ and $S$. 
\end{enumerate}

\begin{proposition}\label{1.33} Let $R\subset S$ be an integral  (hence finite) t-closed  FCP extension  and  $I:=(R:S)$. Then $V_R(I) = \mathrm{MSupp}(S/R)$ is finite 
and $Q\in \mathrm V_S(I)$ lying over $P$ in $R$ verifies $Q=PS$. 
Moreover,  if  $(\mathcal T)$ is a 
convenient
  property,  the following conditions are equivalent:

\begin{enumerate}

\item $R \subset S$ satisfies $(\mathcal T)$.

\item  The residual extensions $\kappa_R(P)\to \kappa_S(Q)$ satisfy $(\mathcal T)$
 for any  $Q\in\mathrm{V}_S(I)$ and $P:=Q\cap R$. 
 
 \item    $\kappa_R(P)\to \kappa_R(P)\otimes_RS$ satisfies $(\mathcal T)$
 for any  $P\in\mathrm{V}_R(I)$.
 
 \item  The field extension $ R/P\to S/Q$ satisfies $(\mathcal T)$
 for any  $Q\in\mathrm{V}_S(I)$ and $P:=Q\cap R$. 
\end{enumerate}
\end{proposition}

\begin{proof}  Since $R\subset S$ is an integral FCP  extension,   $R/I$ is an Artinian ring \cite[Theorem 4.2]{DPP2} and the extension is finite; so that,  $\mathrm{V}_R(I)=\mathrm{Supp}(S/R)= \mathrm{MSupp}(S/R)$, where the last equation is valid because $R/I$ is  Artinian.  
 Actually $I$ is an intersection of finitely many maximal ideals because the extension is seminormal \cite[Lemma 4.8]{DPP2}. 
Moreover, for each $P\in\mathrm{V}_R(I)$, there is a unique $Q\in\mathrm{V}_S(I)$ lying above $P$ by Proposition~\ref{1.31}(1). In particular, $S_Q=S_P$ and $\kappa_S(Q)=S_Q/QS_Q=S_P/QS_P=S_P/PR_P\ (*)$ because, $S_P=S_Q$ is a local ring with  maximal ideal $PR_P$ according to \cite[Lemma 3.17]{DPP3}.
 This shows that $PR_P=(R_P:S_P)=PS_P=QS_Q=QS_P$. Since $P_MS_M=Q_M=S_M$ for any $M\in\mathrm{Max}(R)\setminus\{P\}$, we get $PS=Q$.
 To conclude, $\kappa_S(Q)=\kappa_R(P)\otimes_RS$,  which gives  (2) $\Leftrightarrow$ (3). 

Because $(\mathcal T)$ is  convenient,  $R\subset S$ satisfies $(\mathcal T)$ if and only if   $R_M\subset S_M$ satisfies $(\mathcal T)$ for any $M\in\mathrm{MSupp}(S/R)$. We have just seen that $(R_M:S_M)=MR_M$, which is an ideal shared by $R_M$ and $S_M$. Fix $P\in\mathrm{MSupp}(S/R)$. Using the second equivalence characterizing a convenient property and the equalities of $(*)$, we get that $R_P\subset S_P$ satisfies $(\mathcal T)$ if and only if   $\kappa_R(P)=R_P/PR_P\subset S_P/PR_P=\kappa_S(Q)$ satisfies $(\mathcal T)$. Gathering these two equivalences, we get that  (1) is equivalent  (2).

 Moreover, since $R/P\cong R_P/PR_P=\kappa_R(P)$ and $S/Q\cong S_Q/QS_Q=\kappa_S(Q)$, we get (2) $\Leftrightarrow$ (4). 
\end{proof} 

This proposition shows that, in order to prove that a convenient  property $(\mathcal T)$ holds  for  integral t-closed FCP extensions, it is enough to establish that this property holds for  finite-dimensional field extensions.
  
\section {First properties of 
 catenarian extensions} 
 
 If $R\subseteq S$ is an FCP extension,  $[R,S]$ is a complete Noetherian Artinian lattice, with $R$ as the least element and $S$ as the largest element. We use lattice definitions and properties described in \cite{NO}.

In the context of a lattice $[R,S]$, some  definitions and properties of lattices have the following formulations.

An element $T$ of $[R,S]$ is an {\it atom}  if and only if $R\subset T$  is a minimal extension. 
 We denote by $\mathcal{A}$ the set of atoms of $[R,S]$. 
 
(a) Following Dobbs and Shapiro in \cite{DS2}, we say that an FCP extension $R\subseteq S$ is {\it catenarian or graded}  if all maximal chains between $R$ and $S$ have the same length (the Jordan-H\"older chain condition).
Actually,  in lattice theory \cite[Theorem 1.12 and the definition just before]{R}, a poset $(P,<)$ is called {\it graded} if there is a function $g:P\to\mathbb{Z}$ for which $a<b\Rightarrow g(a)<g(b)$, with $g(b)=g(a)+1$ if 
 $b$ covers $a$,
  which is equivalent, for an FCP extension to all maximal chain between two comparable elements have the same finite length. Therefore,   $R\subseteq S$ is catenarian if and only  if  $[R,S]$ is a graded lattice. 

(b) {\it distributive} if intersection and product are each distributive with respect to the other. Actually, each distributivity implies the other \cite[Exercise 5, page 33]{NO}.
 
   We begin by giving some simple examples. 

\begin{proposition} \label{1.04} A chained extension of finite length is catenarian.
\end{proposition} 
 \begin{proof} Obvious.
 \end{proof}
 
 \begin{proposition} \label{1.0} \cite[Corollary 14, page 62]{G} A distributive lattice of finite length is catenarian.  
\end{proposition} 
 
  \begin{proposition} \label{1.4} Let $R\subseteq S$ be an FCP extension. Then, $R\subseteq S$ is catenarian  if and only if so is each subextension $T\subseteq U$ of $R\subseteq S$.
\end{proposition} 
 \begin{proof} Obvious.
 \end{proof}
 
\begin{proposition} \label{1.2} A  ring extension $R\subset S$ of length 2 is  catenarian.
\end{proposition}

\begin{proof} Clear, because any maximal chain of $R\subset S$ cannot be minimal and so is of length 2.
\end{proof} 

  We are going to look 
   at 
  the stability of catenarity under  some usual constructions of commutative algebra. 

\begin{proposition}\label{5.12} Let $R \subseteq S$ be an  FCP extension, $J$ an ideal of  $S$ and  $I:=J\cap R$. 
 If $R \subseteq S$ is catenarian, so is $R/I\subseteq S/J$. 
\end{proposition}
\begin{proof}
 Since catenarity still holds for any subextension, it  holds for $R+J\subseteq S$, and also for $(R+J)/J\subseteq S/J$, because of the lattice isomorphism $[R+J,S]\to [(R+J)/J, S/J]$ defined by $T\mapsto T/J$, since $J$ is a  ideal shared by any element of $[R+J,S]$. Now, $(R+J)/J\cong R/I$  gives the result.
\end{proof}

We remark that the converse may not hold. 
 Let $R \subseteq S$ be an  FCP extension, $J$ an ideal of  $S$ and  $I:=J\cap R$. If $R/I\subseteq S/J$ is catenarian, so is $R+J\subseteq S$, (use the lattice isomorphism of the proof of Proposition  \ref{5.12}.) But,   if $R\subseteq R+J$ 
is not catenarian, then 
   $R\subseteq S$ is not. 
 
  \begin{corollary}\label{1.014} Let $R\subseteq S$ be a ring  extension. The following statements are equivalent: 
\begin{enumerate}
\item $R\subseteq  S$ is catenarian.

\item $R/I\subseteq S/I$ is  catenarian  for each  ideal $I$ shared by $R$ and $S$.

\item  $R/I\subseteq S/I$ is  catenarian  for some  ideal $I$ shared by $R$ and $S$.
\end{enumerate}
\end{corollary}
\begin{proof}  (1) $\Rightarrow$ (2) by Proposition \ref{5.12}. Obviously (2) $\Rightarrow$ (3) and (3) $\Rightarrow$ (1) because of the bijection $[R,S]\to [R/I,S/I]$. 
\end{proof}

 Recall that the idealization $R(+)M:=\{(r,m)\mid (r,m)\in R\times M\}$ is a commutative ring whose operations are defined as follows: 

$(r,m)+(s,n)=(r+s,m+n)$ \ \   and  \ \ \ $(r,m)(s,n)=(rs,rn+sm)$

Then  $(1,0)$ is the unit of $R(+)M$, and $R\subseteq R(+)M$ is a ring morphism defining $R(+)M$ as an $R$-module; so that we can identify any $r\in R$ with $(r,0)$. 

 \begin{proposition}\label{5.14} Let   $M$ be an $S$-module. Then an  FCP extension $R \subseteq S$  
  is catenarian if and only if so is
  $R(+)M\subseteq S(+)M$. 
\end{proposition}
\begin{proof}  Since $M$ is an $S$-module, it is also an  $R$-module, and there is  a ring extension $R(+)M\subseteq S(+)M$. Set $C:=(R:S)$. Obviously, we have $C(+)M=(R(+)M:S(+)M)$. Since $(R(+)M)/(C(+)M)\cong R/C$ and $(S(+)M)/(C(+)M)\cong S/C$, the following holds: $R(+)M\subseteq S(+)M$ 
 is catenarian 
$\Leftrightarrow (R(+)M)/(C(+)M)\subseteq (S(+)M)/(C(+)M)$ 
 is catenarian 
$\Leftrightarrow R/C\subseteq S/C$ 
 is catenarian 
$\Leftrightarrow R\subseteq S$ 
 is catenarian   by  
 Corollary \ref{1.014}.
 \end{proof}
 
The next Proposition is needed in  the last example of this section.  

 \begin{proposition}\label{5.13} \cite[Lemma III.3]{DMPP} and \cite[Proposition 3.7(d)]{DPP2} Let $R \subseteq S$ be a ring extension, where $R=\prod_{i=1}^n R_i$ is a  cartesian product of rings. For each $i\in\mathbb N_n$, there are ring extensions $R_i\subseteq S_i$ such that $S= \prod_{i=1}^n S_i$.  
  Each $T\in[R,S]$ is of the form $T=\prod_{i=1}^n T_i$, where $T_i\in[R_i,S_i]$ for each $i\in\mathbb N_n$.
  Moreover  $R \subseteq S$ has FCP
  if and only if 
so has $R_i \subseteq S_i$ for each $i\in\mathbb N_n$. \end{proposition}

\begin{proposition}\label{3.9} Let $R \subseteq S$ be an  FCP extension, where $R=\prod_{i\in\mathbb N_n} R_i$ is a cartesian  product of rings. For each $i\in\mathbb N_n$, there exists FCP ring extensions $R_i\subseteq S_i$ such that $S= \prod_{i\in\mathbb N_n} S_i$. Moreover $R \subseteq S$ is catenarian  if and only if so is  $R_i\subseteq S_i$ for each $i\in\mathbb N_n$. If these conditions hold, then $\ell[R,S]=\sum_{i\in\mathbb N_n}\ell[R_i,S_i]$.
\end{proposition}
\begin{proof} The first part of the statement  is Proposition \ref{5.13}. Set $r:=\ell[R,S]$ and  $r_i:=\ell[R_i,S_i]$.   
 
 Assume  that $R \subseteq S$ is catenarian, and so is any subextension of $R\subset S$. Let $\mathcal C_i:=\{T_{i,j}\}_{j\in \mathbb N_s}$ be a maximal chain of $[R_i,S_i]$. Using the  construction of Proposition \ref{5.13} and setting $T_j:=\prod_{k\in\mathbb N_n} T_{k,j}$ such that $T_{k,j}=R_k$  for $k\neq i$ and $j\in \mathbb N_s$, we get a maximal chain  $\mathcal C:=\{T_j\}_{j\in \mathbb N_s}$ of $[R,T_s]$. Since $R\subset T_s$ is catenarian (of length $s$), any maximal chain of $[R_i,S_i]$ leads, by this construction, to a maximal chain of $[R,T_s]$ of length $s$, and so is of length $s=r_i$. Then, $R_i \subseteq S_i$ is catenarian. 

Conversely, assume that $R_i\subseteq S_i$  is catenarian for each $i\in\mathbb N_n$. Let $\mathcal C:=\{T_j\}_{j\in\mathbb N_r}$ be a maximal chain of $[R,S]$ with $T_{j-1}\subset T_j$  minimal for each $j\in\mathbb N_r$. For each $j\in\mathbb N_r$, set $T_j:=\prod_{i\in\mathbb N_n} T_{i,j}$. For a given $j\in\mathbb N_r$, since $T_{j-1}\subset T_j$ is minimal, there exists  a unique $i_j\in\mathbb N_n$ such that $T_{i_j,j-1}\subset T_{i_j,j}$ is  minimal and $T_{k,j-1}= T_{k,j}$ for any $k\in\mathbb N_n\setminus\{i_j\}$. Then, $\mathcal C_{i_j}:=\{T_{i_j,j}\}_{j\in\mathbb N_r}$ is a maximal chain of $[R_{i_j},S_{i_j}]$ with $T_{i_j,j-1}\subset T_{i_j,j}$   minimal and either $T_{i_j,l-1}= T_{i_j,l}$, or $T_{i_j,l-1}\subset T_{i_j,l}$ minimal for $l\neq j$. Now, 
looking at the elements of $\mathcal C$ and of the $\mathcal C_i$'s, we get that each minimal extension in $\mathcal C$ comes from only one minimal extension in some $\mathcal C_{i_j}$ such that the corresponding $j$th extensions in the $\mathcal C_k$'s for $k\neq i_j$ are trivial. Conversely, fix some $i\in\mathbb N_n$ and  let $T_{i,j}\subset T_{i,k}$ be a  minimal extension  in $\mathcal C_i$ with $j<k$. If $k=j+1$, then $T_{i,j}\subset T_{i,j+1}$ comes from $T_j\subset T_{j+1}$ which is minimal. If $k>j+1$, then $T_j\subset T_k$ is not minimal. But there exists $j'\in\{j,\ldots,k-1\}$ such that $T_{i,j}= T_{i,j'}\subset T_{i,j'+1}= T_{i,k}$, so that $T_{i,j'}\subset T_{i,j'+1}$ comes from $T_{j'}\subset T_{j'+1}$ which is still minimal. Any $T_j\in[R,S]$ can be uniquely expressed as a product $T_{1,j}\times\cdots\times T_{n,j}$ where $T_{i,j}\in[R_i,S_i]$ for each $i \in\mathbb N_n$. Since any minimal extension $T_j\subset T_{j+1}$ comes from a unique minimal extension $T_{i_j,j}\subset T_{i_j,j+1}$ corresponding to  a unique $i_j$, we get that  $\ell(\mathcal C)=\sum_{i\in\mathbb N_n}\ell(\mathcal C_i)=\sum_{i\in\mathbb N_n}r_i$ is a constant and $R\subseteq S$  is catenarian. 

If these conditions hold, then $\ell[R,S]=\sum_{i\in\mathbb N_n}\ell[R_i,S_i]$ (see \cite[Proposition 3.7 (d)]{DPP2}). In fact, there was a misprint in the quoted reference, where one should replace $\Pi$ with $\sum$ 
 in the length formula.
\end{proof} 
 
 \section{Characterization of
 catenarian ring extensions}

The aim of this section is to characterize catenarian ring extensions. Some special extensions 
are catenarian.  
 
 \begin{proposition}\label{3.1} A  Pr\"ufer FCP extension is catenarian. 
\end{proposition}

\begin{proof}  Use \cite[Theorem 6.3 and Proposition 6.12]{DPP2}.
\end{proof} 

The following Theorem reduces the study to the charcterization of integral catenarian extensions.  

\begin{theorem}\label{3.2} An FCP extension $R\subseteq S$  is catenarian if and only if $R\subseteq \overline R$  is catenarian.
\end{theorem}

\begin{proof} One implication comes from Proposition \ref{1.4}. 

Conversely, assume that $R\subseteq \overline R$  is catenarian. Set $\ell[R,\overline R]=m$, which is the length of any maximal chain from $R$ to $\overline R$ and $\ell[\overline R,S]=r$, which is the length of any maximal chain from $\overline R$ to $S$ by Proposition  \ref{3.1}. Let $\mathcal C$ be a maximal chain from $R$ to $S$. Using the proof of \cite[Theorem 4.11]{DPP3}, there exists a maximal chain $\mathcal C'$ from $R$ to $S$ containing $\overline R$ with $\ell(\mathcal C')=\ell(\mathcal C)$. Since $\ell(\mathcal C')=\ell[R,\overline R]+\ell[\overline R,S]=m+r$, we get that all maximal chains from $R$ to $S$ have the same length, and $R\subseteq S$  is catenarian.
\end{proof} 

 The previous Proposition shows that we can limit to characterize integral extensions. Before, we show it is enough to deal with  ring extensions $R\subset S$, where $R$ is a local ring.

\begin{proposition}\label{3.3} Let $R\subseteq S$ be an FCP ring  extension.
\begin{enumerate} 
\item If $R\subseteq S$ is  integral, then $R\subseteq S$  is catenarian if and only if $R_M\subseteq S_M$  is catenarian for each $M\in\mathrm{MSupp}(S/R)$.

\item If $R\subseteq S$ is an arbitrary FCP ring  extension, the following statements are equivalent: 
\begin{enumerate}
\item $R\subseteq  S$ is catenarian;

\item $R_M \subseteq S_M$ is catenarian for each $M\in\mathrm{MSupp}(S/R)$;

\item $R_P \subseteq S_P$ is catenarian for each $P\in\mathrm{Supp}(S/R)$.
\end{enumerate}
\item In particular, if $R\subseteq S$ is integral, then

 $\ell[R,S]=\sum_{M\in\mathrm{MSupp}(S/R)}\ell[R_M,S_M]$.
 \end{enumerate}
\end{proposition} 

\begin{proof} 
 We begin to prove the result for an integral extension, since the integral case is involved in the proof of the general case.

(1) Assume that $R\subseteq S$ is catenarian. For some $M\in\mathrm{MSupp}(S/R)$, let $\mathcal C:=\{T_i\}_{i=0}^n$ and $\mathcal C':=\{T'_j\}_{j=0}^m$ be two maximal chains of $[R_M,S_M]$, so that $T_{i-1}\subset T_i$  is minimal for each $i\in\mathbb N_n$, Ê$T'_{j-1}\subset T'_j$ is minimal for each $j\in\mathbb N_m$, with $T_0=T'_0=R_M$ and $T_n=T'_m=S_M$. Using \cite[Theorem 3.6]{DPP2}, for each $i\in\mathbb N_n$, there exists $R_i\in[R,S]$ such that $(R_i)_M=T_i$ and $(R_i)_{M'}=R_{M'}$ for any $M'\in\mathrm{MSupp}(S/R)\setminus\{M\}$. In particular,  $R_{i-1}\subset R_i$ is minimal for each $i\in\mathbb N_n$ by Theorem \ref{crucial}. It follows that $\{R_i\}_{i=0}^n$ is a maximal chain of $[R,R_n]$ of length $n$. In the same way, we can build a maximal chain $\{R'_j\}_{j=0}$ of $[R,R'_m]$  of length $m$ such that $(R'_j)_M=T'_j$ and $(R'_j)_{M'}=R_{M'}$ for any $M'\in\mathrm{MSupp}(S/R)\setminus\{M\}$. But $(R_n)_M=T_n=S_M=T'_m=(R'_m)_M$ and $(R_n)_{M'}=R_{M'}=(R'_m)_{M'}$  for any $M'\in\mathrm{MSupp}(S/R)\setminus\{M\}$. It follows that $R_n=R'_m$. Since $R\subset R_n$ is catenarian by Proposition \ref{1.4}, we get that $m=n$, so that $\ell(\mathcal C)=\ell(\mathcal C')$, and $R_M\subseteq S_M$ is catenarian.

Conversely, assume that $R_M\subseteq S_M$  is catenarian for each $M\in\mathrm{MSupp}(S/R)$. Let $\mathcal C$ and $\mathcal C'$ be two maximal chains of $[R,S]$. We  denote by $\mathcal{C}_M$ (resp. $\mathcal{C}'_M$) the chain deduced from $\mathcal{C}$ (resp. $\mathcal{C}'$) by localization at $M \in \mathrm{Spec}(R)$ and deletion of redundant elements, so that $\mathcal{C}_M$ and  $\mathcal{C}'_M$ are maximal chains of $[R_M,S_M]$, and then $ \ell(\mathcal C_M)=\ell(\mathcal C'_M)$. From \cite[Lemma 4.5]{DPP3}, we deduce that  $\ell(\mathcal C)=\sum_{M\in\mathrm{MSupp}(S/R)}\, \ell(\mathcal C_M)=\sum_{M\in\mathrm{MSupp}(S/R)}\, \ell(\mathcal C'_M)=\ell(\mathcal C')$ and $R\subseteq S$  is catenarian.

(2) According to Theorem \ref{3.2} and (1), we have  (a) $\Leftrightarrow R\subseteq \overline R$  is catenarian $\Leftrightarrow R_M\subseteq \overline R_M$ is catenarian for each $M\in\mathrm{MSupp}(S/R)\Leftrightarrow R_P\subseteq \overline R_P$ is catenarian for each $P\in\mathrm{Supp}(S/R)\Leftrightarrow R_P\subseteq S_P$ is catenarian for each $P\in\mathrm{Supp}(S/R)\Leftrightarrow R_M\subseteq S_M$ is catenarian for each $M\in\mathrm{MSupp}(S/R)$.

 (3) The  equality comes from \cite[Proposition 4.6]{DPP3}.
\end{proof} 

\begin{proposition} \label{desc 1} Let $R\subset S$  be a catenarian extension, $f: R \to R'$ a flat ring epimorphism  and $S':=  R'\otimes_RS$. 
 Then   $R'\subset S'$ is catenarian.
  \end{proposition}
  
  \begin{proof}  
The proof is a consequence of the following facts.  Let $f: R\to R'$ be a flat epimorphism and  $Q \in \mathrm{Spec}(R')$, lying over $P$ in $R$, then $R_P \to R'_Q$ is an isomorphism. Moreover,  we have 
  $(R'\otimes_RS)_Q \cong  R'_Q\otimes_{R_P} S_P$, so that $R_P\to S_P$ identifies to $R'_Q \to  (R'\otimes_RS)_Q=S'_Q $. 
  
  Assume that $ R \subset S$  is  catenarian. Then, so is $R_P\to S_P$ for each $P \in \mathrm{Spec}(R)$ by Proposition \ref{3.3}.  Let $Q \in \mathrm{Spec}(R')$ and $P:=f^{-1}(Q)\in \mathrm{Spec}(R)$. Since $R_P\to S_P$ identifies to $R'_Q \to  S'_Q $, we get that $ R'_Q \subset S'_Q$  is  catenarian  for each $Q \in \mathrm{Spec}(R')$. It follows that $ R' \subset S'$  is catenarian  by the same references.
\end{proof} 

 Given a ring $R$, recall that its  {\it Nagata ring} $R(X)$ is the localization $R(X) = T^{-1}R[X]$ of the ring of polynomials $R[X]$ with respect to the multiplicatively closed subset $T$ of all  polynomials with  content $R$.
 In \cite[Theorem 32]{DPP4}, Dobbs and the authors proved that when $R\subset S$ is an extension, whose Nagata extension $R(X)\subset S(X)$  has FIP, the map $\varphi:[R,S]\to [R(X), S(X)]$ defined by $\varphi(T)= T(X)$ is an order-isomorphism. In this case, we have the following result. 

 \begin{proposition}\label{4.1999} Let  $R\subset S$ be an FCP extension whose Nagata extension $R(X)\subset S(X)$ has FIP. Then,  $R\subset S$ has FIP and $R\subset S$  is catenarian if and only if  $R(X)\subset S(X)$ is catenarian.
 \end{proposition}

\begin{proof}  Since $R(X)\subset S(X)$ has FIP, $R\subset S$ has FIP and the map $\varphi:[R,S]\to[R(X),S(X)]$ defined by $\varphi(T)=T(X)$  is an order-isomorphism by \cite[Theorem 32]{DPP4}. In particular, any maximal chain $\mathcal C:=\{R_i\}_{i=0}^n$ of $[R,S]$ of length $n$ is such that $\varphi(\mathcal C):=\{\varphi(R_i)\}_{i=0}^n$ is  a maximal chain of $[R(X), S(X)]$ of length $n$. Conversely, any maximal chain of $[R(X), S(X)]$ of length $n$ comes through $\varphi$ from 
 a maximal chain  of $[R,S]$ of length $n$. Then, the catenarity of $R\subset S$ is equivalent to the catenarity of $R(X)\subset S(X)$. By the way, we recover the equality $\ell[R,S]=\ell[R(X),S(X)]$ of \cite[Theorem 32]{DPP4} which always holds when $R(X)\subset S(X)$ has FIP.
\end{proof} 

When looking at the previous Proposition, where $R\subset R(X)$ is a faithfully flat ring morphism, we may ask if the statement of the above  Proposition  is still satisfied when taking $R\to R'$ an arbitrary faithfully flat ring morphism instead of   $R\to R(X)$ (which is faithfully flat). The answer is no, as shown  at the end of  the next section (Example \ref{3.10}).

In \cite{Pic 11}, we studied the Loewy series $\{S_i\}_{i=0}^n$  associated to an FCP ring extension $R\subseteq S$ defined as follows in   \cite[Definition 3.1]{Pic 11}: the {\it socle} of the extension $R\subset S$ is  $\mathcal S[R,S]:=\prod_{A\in \mathcal A}A$ and the {\it Loewy series} of the extension $R\subset S$ is the chain $\{S_i\}_{i=0}^n$ defined by induction: $S_0:=R,\ S_1:=\mathcal S[R,S]$ and for each $i\geq 0$ such that $S_i\neq S$, we set $S_{i+1}:=\mathcal S[S_i,S]$. Of course, since $R\subset S$ has FCP, there is some integer $n$ such that $S_n=S_{n+1}=S$. We introduced a property often involved in  \cite{Pic 11}: An FCP 
extension  $R\subset S$ with Loewy series $\{S_i\}_{i=0}^n$ is said to satisfy the  property $(\mathcal P)$ (or is a $\mathcal P$-extension) if $[R,S]=\cup_{i=0}^{n-1}[S_i, S_{i+1}]$. For such extensions, the Loewy series  
 gives a catenarity criterion.  
 For instance, we have a catenarian $\mathcal P$-extension which is not distributive in \cite[Example 3.20]{Pic 11}: 

    Set $k:=\mathbb{Q}, \ L:=k[x]$, where $x:=\sqrt 3+\sqrt 2$ and    $k_i:=k[\sqrt i],\ i=2,3,6$. Then, $[k,L]=\{k,k_1,k_2,k_3,L\}$ and  the following diagram holds:
$$\begin{matrix}
   {}  &        {}      & L             &       {}       & {}     \\
   {}  & \nearrow & \uparrow  & \nwarrow & {}     \\
k_1 &       {}       & k_2         &      {}        & k_3 \\
  {}   & \nwarrow & \uparrow & \nearrow & {}     \\
  {}   &      {}        & k             & {}             & {} 
\end{matrix}$$
 $k\subset L$ is a non-distributive extension of length 2, 
 but is
  catenarian. Moreover, $\mathcal S[k,L]=L=S_1,\ [k,L]=[k,S_1]$ and $k\subset L$ is a $\mathcal P$-extension. 

 \begin{proposition}\label{1.0143}  Let $R\subseteq S$ be an FCP   $\mathcal P$-extension with Loewy series $\{S_i\}_{i=0}^n$. Then $R\subseteq  S$ is catenarian if and only if $S_i\subset S_{i+1}$ is catenarian for each $i\in\{0,\ldots,n-1\}$. Moreover, $\ell[R,S]=\sum_{i=0}^{n-1}\ell[S_i,S_{i+1}]$.
\end{proposition}
\begin{proof} One implication is obvious  because of Proposition \ref{1.4}. Conversely, assume that each  $S_i\subset S_{i+1}$ is catenarian  and set $n_i:=\ell[S_i,S_{i+1}]$. Let $\mathcal C:=\{R_j\}_{j=0}^m$ be a maximal chain of $[R,S]$ so that $R_0=R,\ R_m=S$ and $R_j\subset R_{j+1}$ is minimal for each $j\in\{0,\ldots,m-1\}$. Since $R\subseteq S$ is a   $\mathcal P$-extension, for each $j\in\{0,\ldots,m-1\}$, there is a unique $i_j\in\{0,\ldots,n-1\}$ such that $R_j\in[S_{i_j},S_{i_j+1}[$, so that any $R_j$ is comparable to any $S_i$. We claim that for each $i\in\{0,\ldots,n-1\}$, there is a unique $j_i\in\{0,\ldots,m-1\}$ such that $S_i=R_{j_i}$. For $i\in\{1,\ldots,n-1\}$, set $j_i:=\inf\{j\in\{1,\ldots,m-1\}\mid S_i\subseteq R_j\}$, so that $S_i\subseteq R_{j_i}$. If $S_i=R_{j_i}$, we are done. Assume that $S_i\subset R_{j_i}$.  By definition of $j_i$, we get that $R_{j_i-1}\not\in[S_i,S]$, so that $R_{j_i-1}\subset S_i\subset R_{j_i}$, a contradiction since $R_{j_i-1}\subset R_{j_i}$ is minimal. In particular, $\ell[R_{j_{i-1}},R_{j_i}]=\ell[S_{i-1},S_i]=n_{i-1}$, which leads to $\ell(\mathcal C)=m=\sum_{i=1}^n\ell[R_{j_{i-1}},R_{j_i}]=\sum_{i=1}^n\ell[S_{i-1},S_i]=\sum_{i=1}^nn_{i-1}$, which is independent of $\mathcal C$. Then, any maximal chain of $[R,S]$ has the same length and $R\subset S$ is catenarian.
\end{proof}

Before   characterizing  catenarian integral FCP extensions, we begin to look at special simpler cases.

\begin{proposition} \label{1.5} An infra-integral FCP   extension  is  catenarian.
\end{proposition}
\begin{proof} \cite[Lemma 5.4]{DPP2}.
\end{proof} 

\begin{corollary} \label{1.6} Let $R$ be a commutative ring and $n\geq 2$ a positive integer. Then $R\subseteq R^n$  is  catenarian.
\end{corollary}
\begin{proof} $R\subseteq R^n$  is   infra-integral \cite[Proposition 1.4]{Pic 9}.
\end{proof}

\begin{proposition}\label{3.411} Let $R\subseteq S$ be an integral FCP extension such that ${}_S^tR\subseteq S$ is catenarian. Let $m:=\ell[R,{}_S^tR]$ and $r:=\ell[{}_S^tR,S]$. The following statements hold:

\begin{enumerate}
\item Any maximal chain containing ${}_S^tR$ has length $m+r$.

\item If   $[R,S]=[R,{}_S^tR]\cup[{}_S^tR,S]$, then $R\subseteq S$ is catenarian.
\end{enumerate}
\end{proposition}

\begin{proof}  Proposition  \ref{1.5} ensures us that $R\subseteq {}_S^tR$ is catenarian. If ${}_S^tR\in\{R,S\}$, then $R\subseteq S$ is catenarian. Assume that ${}_S^tR\not\in\{R,S\}$.

(1) Let $\mathcal C$ be a maximal chain of $[R,S]$ containing ${}_S^tR$. Set $\mathcal C_1:=\mathcal C\cap[R,{}_S^tR]$ and $\mathcal C_2:=\mathcal C\cap [{}_S^tR,S]$. Since ${}_S^tR\in \mathcal C$, we get that $\mathcal C_1$ is a maximal chain of $[R,{}_S^tR]$ and $\mathcal C_2$ is a maximal chain of $[{}_S^tR,S]$, so that  $\ell(\mathcal C)=\ell(\mathcal C_1)+\ell(\mathcal C_2)=m+r$.

(2) Assume that  $[R,S]=[R,{}_S^tR]\cup[{}_S^tR,S]$. Let $\mathcal C:=\{R_i\}_{i=0}^n$ be a maximal chain of $[R,S]$ such that $R_{i-1}\subset R_i$  is minimal for each $i\in\mathbb N_n$. Let $k:=\sup\{i\in\mathbb N_n\mid R_i\in[R,{}_S^tR]\}$. In fact, $k\neq n$ because $ {}_S^tR\neq S$. Then
 $R_k\subseteq {}_S^tR\subset R_{k+1}$. Since $R_k\subset R_{k+1}$  is minimal, it follows that ${}_S^tR=R_k\in \mathcal C$. Then (1) yields  that  $\ell(\mathcal C)=m+r$, which  shows that $R\subseteq S$ is catenarian.
\end{proof} 

\begin{lemma}\label{3.41} Let $R\subset T$ and $T\subset S$ be two minimal extensions such that $R\subset T$ is inert and $T\subset S$ is either decomposed or ramified (i.e. non-inert). Setting  $T':={}_S^tR$, the following statements hold: 
\begin{enumerate}
\item If $(R:T)\neq (T:S)$, then $R\subset T'$ (resp. $T'\subset S$) is minimal of the same type as $T\subset S$ (resp. $R\subset T$) and $[R,S]=\{R,T,T',S\}$.

\item If $(R:T)=(T:S)$, then $\ell[R,S]>2$.
\end{enumerate}
\end{lemma} 

\begin{proof} Set $M:=(R:T)$ and $N:=(T:S)$. Since $R\subset T$ is inert, then $M\in\mathrm{Max}(T)$.

(1) If $M\neq N$, then $R\subset T'$ is a minimal extension of the same type as $T\subset S$ and $ T'\subset S$ is a minimal inert extension by the Crosswise exchange (Lemma  \ref{1.13}). Indeed, $N\cap R\not\subseteq M$ since $N\cap R\in\mathrm{Max}(R)$. Moreover, the Crosswise exchange lemma says that $[R,S]=\{R,T,T',S\}$.

 (2) Assume that $M=N$, so that $M=(R:S)$. Since $R\subset S$ is neither infra-integral, nor t-closed, we get that $T'\neq R,S$. Moreover, $R\subset S$ has FCP by \cite[Theorem 4.2]{DPP2}. Then, there exists $R_1\in[R,T']$ such that $R\subset R_1$ is minimal 
  non-inert 
 with $M=(R:R_1)$. Using \cite[Proposition 7.1 or 7.4]{DPPS}, we obtain two maximal chains of $[R,TR_1]$ of different lengths, so that $\ell[R,S]>2$.
\end{proof}

\begin{proposition}\label{3.42} Let $R\subset S$ be an integral FCP extension. For any maximal chain $\mathcal C$ of $[R,S]$, there exists a maximal chain $\mathcal C'$ of $[R,S]$ containing ${}_S^tR$ such that $\ell(\mathcal C')\geq \ell(\mathcal C)$. In particular, $\ell(\mathcal C')> \ell(\mathcal C)$ if and only if there exist $U,T,V\in \mathcal C$ such that $U\subset T$ is inert, $T\subset V$ is minimal 
  non-inert, 
 and $(U:T)=(T:V)$. If these last conditions hold, then $(T:V)\cap R\in \mathrm{MSupp}({}_S^tR/R)$ and $\{P\cap {}_S^tR\mid P\in \mathrm{V}_S((U:T))\}\subseteq \mathrm{MSupp}_{{}_S^tR}(S/{}_S^tR)$. 
\end{proposition} 

\begin{proof} Let $\mathcal C$ be the maximal chain $R= R_0\subset\ldots\subset R_i \subset\ldots\subset R_n=S$, where $R_{i-1}\subset R_i$ is  minimal for each $i\in\mathbb{N}_n$. 

If there exists some $i$ such that $R_i={}_S^tR$, take $\mathcal C'=\mathcal C$.

Assume that $R_i\neq{}_S^tR$ for each $i\in\{0,\ldots,n\}$. In particular, ${}_S ^tR\neq R,S$, so that $R\subset S$ is neither t-closed, nor infra-integral. We claim that there exists some $i$ such that $ R_{i-1} \subset R_i$ is inert and $R_i\subset R_{i+1}$ is minimal 
  non-inert. 
 Deny, so that no inert  extension is followed by a minimal 
   non-inert
   extension. It follows that there exists $k\in\mathbb{N}_{n-1}$ such that $R_j\subset R_{j+1}$ is not inert for $j<k$ and $R_j\subset R_{j+1}$ is inert for $j\geq k$. Then, $R\subset R_k$ is infra-integral and $R_k\subset S$ is t-closed 
 according to Proposition ~\ref{1.31}. 
  It follows  that $R_k={}_S^tR$, a contradiction. An application of Lemma~\ref{3.41} gives a maximal chain from $R_{i-1}$ to $R_{i+1}$ containing at less two minimal extensions and with first, only   minimal 
    non-inert 
    extensions followed by only  inert  extensions.  Repeating this construction until we have first only minimal 
      non-inert
      extensions followed by only inert extensions, we get a maximal chain $\mathcal C'$ of $R$-subextensions of $S$ containing ${}_S^tR$ such that $\ell(\mathcal C')\geq \ell(\mathcal C)$.

Still using   Lemma~\ref{3.41}, we get that $\ell(\mathcal C')> \ell(\mathcal C)$ if and only if  there exist $U,T,V\in \mathcal C$ such that $U\subset T$ is inert, $T\subset V$ is minimal 
 non-inert,
  and $(U:T)=(T:V)$. Assume that these conditions hold, so that there exist $U',V'\in[U,V]$ such that $U\subset U'$ is minimal 
    non-inert
    and $V'\subset V$ is inert. In particular, the conductors of  these minimal extensions lie  above $M:=(U:T)=(T:V)=(U:V)$ in $U$. Using the previous construction to get a maximal chain $\mathcal C'$ of $R$-subextensions of $S$ containing ${}_S^tR$, we get that $N:=M\cap R\in \mathrm{MSupp}({}_S^tR/R)$. Indeed, $R_N\subset V_N$ is neither infra-integral, nor t-closed, so that $({}_S^tR)_N\neq R_N,S_N$. Moreover, since $U\subset T$ is inert, we get that $U/M\not\cong T/M$. For the same reason, let $P\in\mathrm{Max}(S)$ be lying over $M$, so that $R/(M\cap R)\not\cong S/P$, giving $P\cap {}_S^tR\in \mathrm{MSupp}_{{}_S^tR}(S/{}_S^tR)$, because $U_M\subset S_M$ is not infra-integral. Then,   $\{P\cap {}_S^tR\mid P\in \mathrm{V}_S(M)\}\subseteq \mathrm{MSupp}_{{}_S^tR}(S/{}_S^tR) $.
\end{proof}

 Let $R\subseteq S$ be a ring extension. Then  $R$ is said {\it unbranched } in $S$ if $\overline R$ is local. We recall the following Lemma.
 
\begin{lemma} \label{3.43} \cite[Lemma 3.29]{Pic 11} Let $R \subset S$ be an  FCP ring extension such that $R$ is unbranched  in $S$. Then, $T$ is local for each $T\in [R,S]$.
\end{lemma} 

\begin{theorem}\label{3.5}  An integral FCP unbranched  extension $R\subseteq S$ is catenarian if and only if $[R,S]=[R,{}_S^tR]\cup[{}_S^tR,S]$ and ${}_S^tR\subseteq S$ is catenarian. 
\end{theorem}

\begin{proof} Since $R$ is unbranched  in $S$, any element of $[R,S]$ is local by Lemma \ref{3.43}. 

One implication of the statement is Proposition  \ref{3.411}. Conversely, assume that $R\subseteq S$ is catenarian. Then ${}_S^tR\subseteq S$ is catenarian by Proposition \ref{1.4}. We claim that $[R,S]=[R,{}_S^tR]\cup[{}_S^tR,S]$. Deny, so that there exists some $T\in [R,S]\setminus [R,{}_S^tR]\cup[{}_S^tR,S]$. In particular, $R\subset T$ is not infra-integral and $T\subset S$ is not t-closed. Then $T':={}_T^tR\neq T$ and  $T'\subset T$ is t-closed. Since $R\subseteq S$ is FCP, there exist $U\in[T',T],\ V\in [T,S]$ such that $U\subset T$ is minimal inert and $T\subset V$ is minimal ramified (it cannot be minimal decomposed, because $V$ is local). Let $M$ be the maximal ideal of $U$, so that $M=(U:T)\in\mathrm {Max}(T')$, because $T'\subseteq U$ is t-closed \cite[Lemma 3.17]{DPP3},  giving $M=(T':V)$ because $M\in\mathrm {Max}(T)$ and $M=(T:V)$. Then Lemma   \ref{3.41} shows that $\ell[U,V]>2$ which yields that $U\subset V$ is not catenarian, a contradiction. To conclude, $[R,S]=[R,{}_S^tR]\cup[{}_S^tR,S]$.
\end{proof} 

\begin{proposition}\label{3.6}  Let $R\subseteq S$ be an integral FCP extension such that ${}_S^tR\subseteq S$ is catenarian. Setting  $\mathcal M_1:=\{N\in\mathrm {Max}(S)\mid N\cap R\in\mathrm {MSupp}({}_S^tR/R)\}$ and $\mathcal M_2:=\{N\in\mathrm {Max}(S)\mid N\cap R\in\mathrm{MSupp}(S/{}_S^tR)\}$, the following statements hold: 
\begin{enumerate}
\item  $\mathcal M_1\cup\mathcal M_2=\{N\in\mathrm {Max}(S)\mid N\cap R\in\mathrm {MSupp}(S/R)\}$.

\item $R\subseteq S$ is catenarian if and only if, for each $U,T,V\in[R,S]$ such that $U\subset T$ is minimal inert and $T\subset V$ is minimal 
  non-inert,
  and for each $N\in\mathcal M_1, N'\in\mathcal M_2$ such that $(U:T)=N'\cap T$ and $(T:V)=N\cap T$, we have $N\neq N'$. 
\item If $\mathcal M_1\cap\mathcal M_2=\emptyset$, then $R\subseteq S$ is catenarian.
\end{enumerate}
\end{proposition}

\begin{proof} (1) Obvious since $N\cap R\in\mathrm {MSupp}(S/{}_S^tR)$ for any $N\in\mathcal M_2$ and $\mathrm {MSupp}(S/R)=\mathrm {MSupp}({}_S^tR/R)\cup\mathrm {MSupp}(S/{}_S^tR)$.

(2) Set $n:=\ell[R,S]$. 
By  Propositions   \ref{3.411} and  \ref{3.42}, any chain containing ${}_S^tR$ has length $n$ and   $R\subseteq S$ is not catenarian if and only if there exist $U,T,V\in[R,S]$ such that $U\subset T$ is  inert and $T\subset V$ is minimal 
  non-inert,
  with $(U:T)=(T:V)$. Assume that  $R\subseteq S$ is not catenarian. Then,  $(T:V)\cap R\in \mathrm{MSupp}({}_S^tR/R)$ and $\{N\cap {}_S^tR\mid N\in \mathrm{V}_S((U:T))\}\subseteq\mathrm{MSupp}_{{}_S^tR}(S/{}_S^tR)$. Let $N\in\mathrm{V}_S((U:T))$. It follows that $N\cap R=(T:V)\cap R\in \mathrm{MSupp}({}_S^tR/R)$, so that $N\in \mathcal M_1$. Moreover, $N\cap {}_S^tR\in\mathrm{MSupp}_{{}_S^tR}(S/{}_S^tR)$ implies that $N\cap R\in\mathrm{MSupp}_R(S/{}_S^tR)$, so that $N\in \mathcal M_2$. Then, $N\in \mathcal M_1\cap\mathcal M_2$. Conversely, if there exists some $N\in \mathcal M_1\cap\mathcal M_2$, such that $(U:T)=N\cap T$ and $(T:V)=N\cap T$ for some $U,T,V\in[R,S]$ such that $U\subset T$ is minimal inert and $T\subset V$ is minimal 
    non-inert,
    then $(U:T)=(T:V)$ so that  $R\subseteq S$ is not catenarian. 
 
 To conclude, 
 $R\subseteq S$ is catenarian if and only if, for each $U,T,V\in[R,S]$ such that $U\subset T$ is minimal inert and $T\subset V$ is minimal 
   non-inert,
   and for each $N\in\mathcal M_1, N'\in\mathcal M_2$ such that $(U:T)=N'\cap T$ and $(T:V)=N\cap T$, we have $N\neq N'$.
 
 (3) If $\mathcal M_1\cap\mathcal M_2=\emptyset$, then (2) shows that $R\subseteq S$ is catenarian.
\end{proof}

If $R\subset S$ is a ring extension and $n$ a positive integer, an $n$-minimal subextension is an extension $U\subseteq V$ where $U, V \in [R,S]$, which is a tower of $n$ minimal extensions.
We say that $R\subset S$ is $2$-{\it catenarian} 
if each of its  $2$-minimal subextension has length $2$.

 This definition will  allow us to give a simpler characterization of  catenarity than in Proposition \ref{3.6}. 
 Moreover, we gave in \cite{Pic 6} a complete characterization of  extensions of length 2, which play a fundamental  role in the lattice properties of extensions.

\begin{theorem}\label{3.62}   An integral FCP extension $R\subseteq S$  is catenarian  if and only if ${}_S^tR\subseteq S$ is catenarian   and $R\subseteq S$ is $2$-catenarian. 
\end{theorem}
\begin{proof} One implication is obvious. Conversely, assume that 
${}_S^tR\subseteq S$ is catenarian,
 $R\subseteq S$ is $2$-catenarian but not catenarian. The assumption of the Theorem is the same as in 
  the proof of Proposition \ref{3.6}. Using this
   proof and Proposition \ref{3.42}, $R\subseteq S$ is not catenarian if and only if  there exist $U,T,V\in[R,S]$ such that $U\subset T$ is  inert and $T\subset V$ is minimal  non-inert, with   $(U:T)=(T:V)$. For such $U,T,V$, set $M:=(U:T)=(T:V)$, so that $M=(U:V)\in\mathrm{Max}(U)$. Because $R\subseteq S$ is $2$-catenarian, we get that $\ell[U,V]=2$. Since $U\subset V$ is neither t-closed, nor infra-integral, $T':={}_V^tU\neq U,V$. Let $U_1\in[U,T']$ be such that $U\subset U_1$ is minimal either ramified or decomposed since $U\subset T'$ is infra-integral. Then, $(U:U_1)=M$ and $TU_1\subseteq V$, which implies that $\ell[U,V]=2\geq\ell[U,TU_1]>2$ according to \cite[Propositions 7.1 and 7.4]{DPPS}, a contradiction. It follows that $R\subseteq S$ is  catenarian.
\end{proof}

\begin{corollary}\label{3.63}   An  integral FCP extension $R\subseteq S$ such that   ${}_S^tR\subseteq S$ is catenarian   and $[R,S]=[R,{}_S^tR]\cup[ {}_S^tR, S]$ is catenarian. 
\end{corollary}
\begin{proof} According to Theorem \ref{3.62}, it is enough to show that $R\subseteq S$ is $2$-catenarian to get that $R\subseteq S$ is catenarian. So, let $T,U,V\in[R,S]$ such that $T\subset U$ and $U\subset V$ are minimal. Then, either $U\in[R,{}_S^tR[\ (*)$ or $U={}_S^tR\ (**)$ or $U\in] {}_S^tR, S]\ (***)$. In case $(*)$, we have $T,U,V\in[R,{}_S^tR]$, because $U\subset V$ is minimal, so that $V\not\in] {}_S^tR, S]$ since $U\subset {}_S^tR$. Then, $\ell[T,V]=2$ by Proposition \ref{1.5}. In case $(**)$, we have  $U={}_S^tR$. It follows that $T\in[R,{}_S^tR[$ and $V\in]{}_S^tR,S]$. We claim that $[T,V]=\{T,U,V\}$. Deny, and let $U'\in[T,V]\setminus\{T,U,V\}$. Then, $U'\neq {}_S^tR$, which implies $U'\in[R,U[\cup] U, S]$. If $U'\in[R,U[$, we have $T\subset U'\subset U$, a contradiction since $T\subset U$ is minimal. If $U'\in]U,S]$, we have $U\subset U'\subset V$,  a contradiction since $U\subset V$ is minimal. Then, $[T,V]=\{T,U,V\}$ and $\ell[T,V]=2$. In case $(***)$,  we have $T,U,V\in[{}_S^tR,S]$, because $T\subset U$ is minimal, so that $T\not\in[R, {}_S^tR[$ since $ {}_S^tR\subset U$. Then, $\ell[T,V]=2$ because ${}_S^tR\subseteq S$ is catenarian. Gathering all the situations, we get that $R\subseteq S$ is $2$-catenarian, and then catenarian.
\end{proof}

 \begin{remark}\label{3.61} We can easily find examples of catenarian FCP  integral extensions $R\subset S$ such that ${}_S^tR\neq R,S$ with ${}_S^tR\subset S$  catenarian and $\mathcal M_1\cap\mathcal M_2\neq\emptyset$ in  Proposition  \ref{3.6}. It is enough to take $(R,M)$ a local ring, so that $\mathcal M_1=\mathrm {Max}(S)=\mathcal M_2$ since any $N\in\mathrm {Max}(S)$ satisfies $N\cap R=M\in\mathrm {MSupp}({}_S^tR/R)\cap\mathrm{MSupp}(S/{}_S^tR)$. See \cite[Proposition 4.7]{Pic 6} and, for the example \cite[Example 4.10(2)]{Pic 6}: $K\subset L$ is a minimal field extension. Set $R:=K$ and $S:=K\times L$. Then ${}_S^tR=K^2$ and $[R,S]=\{R,{}_S^tR,S\}$ which is obviously catenarian.
\end{remark} 

The remaining case to study is the characterization of t-closed FCP catenarian extensions. 
For an FCP $t$-closed  integral extension, the catenarian property is a convenient property according to Proposition \ref{3.3} and Corollary \ref{1.014}. Then, using Proposition \ref{1.33}, we get that an FCP integral t-closed  extension $R\subset S$ with $I:=(R:S)$ is catenarian if and only if the residual extensions $R/P= \kappa_R(P)\to \kappa_S(Q)=S/Q$ is catenarian 
 for any  $Q\in\mathrm{V}_S(I)$ and $P:=Q\cap R$. It follows that  the catenarian property relies heavily on the  characterization of catenarian field extensions. This last characterization is an open problem. 
 
  In this paper,   a purely inseparable field extension is called 
{\it radicial}.  Note that  if $k\subset L$ is a radicial FIP field extension, then $[k,L]$ is a chain, whence $K \subseteq L$ is catenarian. 
 
The following result by Dobbs-Shapiro gives a new reduction of the study.

 \begin{proposition}\label{3.7} \cite[Theorem 2.9]{DS2}. A finite-dimensional field extension  $k\subset L$, with separable  closure $T$ is  catenarian   if and only if $k\subseteq T$ is catenarian. In particular, a finite-dimensional radicial field extension  is catenarian.
 \end{proposition} 
 
 In the introduction of their paper  \cite[page 2604]{DS2}, Dobbs and Shapiro noted that a finite-dimensional Galois field extension is catenarian if and only if its Galois group is supersolvable. 
 
Recall that a finite group $G$ is called {\it supersolvable} if there exists a finite chain $\{1\}=G_r\subset \ldots\subset G_1\subset G_0=G$  of normal subgroups of $G$ such that each factor group $G_i/G_{i+1}$ is cyclic. A finite lattice $L$ is called {\it supersolvable} \cite[Example 3.14.4]{St} if it possesses a maximal chain $\mathcal C$ such that the sublattice generated by $\mathcal C$ and any other chain of $L$ is distributive. 
 For a finite Galois extension, we are going to show that being catenarian is equivalent to being supersolvable  as a lattice. 
 To  this aim, we introduce a new material:  An element $x$ of a lattice $(L,\vee,\wedge)$ is {\it left modular} if it satisfies : $(y\vee x)\wedge z=y\vee(x\wedge z)$ for all $y\leq z$, and $L$ is called  {\it left modular} if it has a maximal chain of left modular elements \cite[page 481]{T}.
  The reader is warned that this notion should not be confused   with the classical definition of modular.

\begin{proposition}\label{3.814} \cite[page 2604]{DS2} Let $k\subset L$ be  a finite  Galois extension with Galois group $G$. The following conditions are equivalent: 
 \begin{enumerate}
\item $k\subset L$ is catenarian.
 \item $G$ is supersolvable.
\item $G$ is catenarian.
\end{enumerate}
\end{proposition}

\begin{proposition}\label{3.813}  Let $k\subset L$ be  a finite  Galois  extension. Then, $k\subset L$  is catenarian if and only if 
 $[k,L]$ is   a supersolvable lattice. 
\end{proposition}

\begin{proof} By \cite[page 197]{St1}, a  supersolvable lattice is catenarian. 

Conversely, assume that $k\subset L$ is  catenarian and let $G$ be the Galois group of the extension. In \cite[Example 2.5]{St1}, it was proved that for a supersolvable group $G$, the lattice of its subgroups is a supersolvable lattice. Then Proposition \ref{3.814} asserts that the lattice $\mathcal G$ of subgroups of $G$ is supersolvable. Let $\varphi:[k,L]\to\mathcal G$ defined by $\varphi (K):=$Aut$_K(L)$, the group of $K$-automorphisms of $L$, for each $K\in[k,L]$. 
Then, $\varphi$ 
is a reversing order isomorphism of lattices. Let $\mathcal C_1$ be a maximal chain of $[k,L],\ \mathcal C_2$ be an other chain of $[k,L]$ and $\mathcal H$ the sublattice of $[k,L]$ generated by $\mathcal C_1$ and $\mathcal C_2$, that is the least sublattice of $[k,L]$  containing $\mathcal C_1$ and $\mathcal C_2$, and closed for $\cap$ and $\prod$. Set $\mathcal C'_1:=\varphi(\mathcal C_1)$ and $\mathcal C'_2:=\varphi(\mathcal C_2)$. Then, $\mathcal C'_1$ is a maximal chain of $\mathcal G$ and $\mathcal C'_2$ is a chain of $\mathcal G$. Since $\mathcal G$ is supersolvable, the lattice $\mathcal H'$  of $\mathcal G$ generated by $\mathcal C'_1$ and $\mathcal C'_2$ is distributive, and so is any sublattice of $\mathcal H'$. Because of the reversing order isomorphism $\varphi$, it follows that $\varphi(\mathcal H)$ is the greatest sublattice of $\mathcal G$ contained in  $\varphi(\mathcal C_1)=\mathcal C'_1$ and $\varphi(\mathcal C_2)=\mathcal C'_2$, and closed for $\cap$ and $\prod$. In particular, $\varphi(\mathcal H)\subseteq \mathcal H'$. Then, $\varphi(\mathcal H)$ is distributive, and so is $\mathcal H$. To conclude, $[k,L]$ is supersolvable.
\end{proof}

\begin{proposition}\label{3.81}  Let $k\subset L$ be  a finite  separable extension with normal closure  $N$. Then, $k\subset N$  is catenarian if and only if the Galois group $G$ of $k\subset N$ is  supersolvable. If this condition holds,   $k\subset L$  is catenarian.
\end{proposition}

\begin{proof}  Proposition \ref{3.814} entails the first part of the Proposition. If this condition holds,   $k\subset L$  is catenarian, since so is $k\subset N$. 
\end{proof}

Proposition \ref{3.813} cannot be generalized to any catenarian extension.
 Nevertheless, we have the following result.

\begin{proposition}\label{3.812}  Let $R\subset S$ be  an FIP extension. Then, $R\subset S$ is catenarian and left modular  if and only if $R\subset S$ is supersolvable.
\end{proposition}

\begin{proof} One implication is Thomas paper \cite[Theorem 1.1]{T}. Conversely,  a supersolvable lattice is catenarian \cite[page 197]{St1}. Moreover, \cite[Proposition 2.2]{St1} shows that a supersolvable lattice is left modular.
\end{proof}

We next  give another characterization of  catenarian finite separable field extensions.
 
 For a field $k$, we denote by $k_u[X]$ the set of all monic polynomials of $k[X]$. Our riding hypotheses  are: $L:=k[x]$ is a finite separable (whence FIP) field extension of $k$  with degree $n$ and $f(X)\in k_u[X]$ is {\it the}   
 minimal polynomial of $x$ over $k$.    
  For any $K\in[k,L]$, we denote by $f_K(X)\in K_u[X]$ the minimal polynomial of $x$ over $K$ 
 and set    
 $\mathcal D:=\{f_K\mid K\in[k,L]\}$. The elements of $\mathcal D$ have been characterized in \cite[Proposition 4.13]{Pic 10}. According to  \cite[Corollary 4.9]{Pic 10}, there is a reversing order bijection $[k,L]\to \mathcal D$ defined by $K\mapsto f_K$, which gives that $\mathcal D$ is an ordered set for the divisibility and a lattice. 

We begin with a characterization of minimal subextensions of $[k,L]$.

 \begin{proposition}\label{3.71}  \cite[Proposition 4.8]{Pic 11} Let $L:=k[x]$ be a finite separable field extension of $k$ and $K\subset K'$ a subextension. Then, $K\subset K'$ is minimal if and only if $f_{K'}$ is a maximal proper divisor of $f_K$ in $\mathcal D$.
\end{proposition} 

\begin{proposition}\label{3.8} A finite separable field extension $k\subset L=k[x]$ is catenarian if and only if all maximal chains of   $\mathcal D$ have the same length.
  \end{proposition} 

\begin{proof}  Use the bijection $[k,L]\to \mathcal D$ defined by $K\mapsto f_K$.
\end{proof}

The following result due to Dobbs-Mullins  gives an example of catenarian field extensions.

\begin{proposition}\label{3.82} \cite[Proposition 2.2 (a)]{DM} Let $k\subset L$ be  a finite  
 Abelian
  field extension such that $[L:k]=\prod_{i=1}^sp_i^{e_i}$ is it prime-power factorization. Then, $k\subset L$ is catenarian  with $\ell[k,L]=\sum_{i=1}^se_i$. 
\end{proposition}

  We can find other examples of catenarian or not catenarian field extensions in     Dobbs and Shapiro's paper \cite{DS2}.

\begin {example}\label{3.83} Let $k\subset L $ be a finite field extension and    $n:=[L:k]$. If $n$ is a prime number or the product $pq$ of any (possibly equal) prime numbers $p$ and $q$, then, $k\subset L $ is catenarian (minimal in the first case, and of length 2 in the second case)\cite[Lemma 2.1(c)]{DS2}. If $n=6m$, where $m$ is an even positive integer, there exists a non-catenarian field extension of degree $n$ \cite[Theorem 2.4]{DS2}. 
\end {example}

\begin {example} \label{3.10} Let $k\subset L$ be a finite separable field extension of degree $n$ which is not catenarian (see Example \ref{3.83}) and let $\Omega$ be an algebraic closure of $k$. Then, $k\subset L$ is \'etale and $\Omega$ diagonalizes $L$ \cite[Proposition 2, A V.29]{Bki A}. More precisely, $\Omega\otimes_k L\cong \Omega^n$. According to Corollary \ref{1.6}, $\Omega\subset \Omega^n$ is catenarian, although $k\subset L$ is not, while $k\subset \Omega$ is faithfully flat (See our remark after Proposition \ref{4.1999}).
  \end{example} 
  
  We end this paper with an application of our results to pointwise minimal extensions. We characterized these extensions in a joint paper with Cahen \cite{CPP}. A ring extension $R\subset S$ is called {\it pointwise minimal}  if $R\subset R[t]$ is minimal for each $t\in S\setminus R$. A pointwise minimal extension is either integral or Pr\"ufer, and in this last case, is minimal, and then catenarian. In the integral case, there is a maximal ideal $M$ of $R$ such that $M=(R:S)$. There are four types of integral pointwise minimal extensions. Since we study the catenarian property, we consider only FCP extensions. So, let $R\subset S$ be an integral FCP pointwise minimal extension. Set $M:=(R:S)$ and $k:=R/M$. Then, $R\subset S$ satisfies one of the following conditions \cite[Theorem 5.4]{CPP} (where  ${}_S^+R$ denotes  the Swan seminormalization of $R$ in $S$ \cite{S}):
\begin{enumerate}
\item $R={}_S^+R$ and $S={}_S^tR$ with $S/M\cong k^2$ if $k\not\cong \mathbb Z/2\mathbb Z$, and $S/M\cong k^n$ if $k\cong \mathbb Z/2\mathbb Z$, where $n:=|\mathrm{V}_S(M)|$.

\item $S={}_S^+R$, with  $x^2\in M$ for any $x\in N$, where $N$ is the unique maximal ideal of $S$ lying above $M$.

\item $R={}_S^tR$, so that $S/M$ is a field and either $k\subset S/M$ is a separable minimal extension or $\mathrm{c}(k)=p$, a prime integer and $x^p\in R$ for any $x\in S$.

\item ${}_S^+R={}_S^tR$, with  $x^2\in M$ for any $x\in N$, where $N$ is the unique maximal ideal of $S$ lying above $M$,Ê $\mathrm{c}(k)=p$ is a prime integer and $x^p\in R$ for any $x\in S$.
\end{enumerate} 

In    all cases, except (4), $R\subset S$ is catenarian  \cite[Proposition 5.8]{CPP}. This is obvious for the two first cases since $R\subset S$ is infra-integral (Proposition \ref{1.5}). In the third case, if $k\subset S/M$ is a separable minimal extension,  $R\subset S$ is minimal and then obviously catenarian. In the second subcase, if $U$ is the separable closure of $k\subset S/M$, we have either $k=U$ or $k\subset U$ minimal, so in both subsubcases $k\subset S/M$ is catenarian according to Proposition \ref{3.7}. The last case leads to a non-catenarian extension. Indeed, setting $T:={}_S^tR=R+N,\ l_1:=\ell[T,S]$ and $\alpha p^{l_1}:=\ell[R,T]$,   \cite[Corollary 5.17]{CPP} shows that for each $i\in\{0,\ldots,l_1\}$, there exists a maximal chain from $R$ to $S$ of length $\alpha p^{i}+l_1$; so that, $R\subset S$ is not catenarian.

\end{document}